\documentclass[a4paper,11pt,leqno]{amsart}

\usepackage{graphics}
\usepackage{thmtools}
\usepackage{wasysym}

\usepackage[T1]{fontenc}    % use 8-bit T1 fonts

\usepackage{amsthm}
\usepackage{amsbsy,amsmath,amssymb,amscd,amsfonts}
\usepackage[pagebackref=true]{hyperref}
\usepackage{url}            % simple URL typesetting
\usepackage{booktabs}       % professional-quality tables
\usepackage{nicefrac}       % compact symbols for 1/2, etc.
\usepackage{microtype}      % microtypography

\usepackage{graphicx,float,latexsym,color}
\usepackage[font={scriptsize,it}]{caption}
\usepackage{subcaption}

\usepackage{makecell}

\usepackage[dvipsnames]{xcolor}

\newtheorem{theorem}{Theorem}
\newtheorem{observation}{Observation}
\newtheorem{proposition}{Proposition}

\newtheorem{corollary}{Corollary}
\newtheorem{lemma}{Lemma}
\theoremstyle{remark}

\theoremstyle{definition}

\hypersetup{
    pdftoolbar=true,        % show Acrobat’s toolbar?
    pdfmenubar=true,        % show Acrobat’s menu?
    pdffitwindow=false,     % window fit to page when opened
    pdfstartview={FitH},    % fits the width of 
    colorlinks=true,       % false: boxed links; true: colored links
    linkcolor=OliveGreen,          % color of internal links (change box color with linkbordercolor)
    citecolor=blue,        % color of links to bibliography
    filecolor=black,      % color of file links
    urlcolor=red           % color of external links
}

\usepackage{lineno}

\usepackage{comment}
\usepackage{microtype}

\newcommand\blfootnote[1]{%
  \begingroup
  \renewcommand\thefootnote{}\footnote{#1}%
  \addtocounter{footnote}{-1}%
  \endgroup
}
\begin{document}

\title[A Special Conic Associated w/ the Reuleaux Triangle]{A Special Conic Associated with the\\Reuleaux Negative Pedal Curve}

%\author[L. G. Gheorghe]{Liliana Gabriela Gheorghe}
%\author[D. Reznik]{Dan Reznik}
%\date{June, 2020}

\author[L. G. Gheorghe]{Liliana Gabriela Gheorghe}
\address{Liliana Gabriela Gheorghe,
Universidade Federal de Pernambuco\\
Dept. de Matemática\\
Recife, PE, Brazil}
\email{liliana@dmat.ufpe.br}

\author[D. Reznik]{Dan Reznik}
\address{Dan Reznik,
Data Science Consulting\\
Rio de Janeiro, RJ, Brazil}
\email{dreznik@gmail.com}

\date{August, 2020}

\maketitle

The Negative Pedal Curve of the Reuleaux Triangle w.r. to a pedal point $M$ located on its boundary
consists of two elliptic arcs and a point $P_0$. Intriguingly, the conic passing through the four arc endpoints and by $P_0$ has one focus at $M$. We provide a synthetic proof for this fact using Poncelet's porism, polar duality and inversive techniques. Additional interesting properties of the Reuleaux negative pedal w.r. to
pedal point $M$ are also included.

\section{Introduction}
\label{sec:intro}

\begin{figure}[H]
\centering
\includegraphics[width=.5\textwidth]{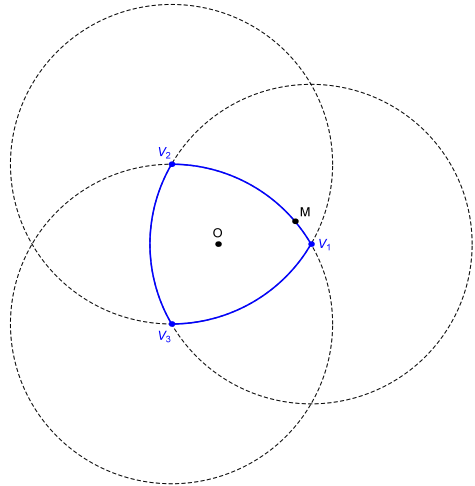}
\caption{The sides of the Reuleaux Triangle $\mathcal{R}$ are three circular arcs  of circles  centered at each Reuleaux vertex $V_i,\;i=1,2,3$. of an equilateral triangle.}
 \label{fig:basic}
\end{figure}

%\begin{figure}
% \centering
% \includegraphics[trim=100 0 30 10,clip,width=.7\textwidth]{pics/0020_reuleaux_envelope.eps}
% \caption{The negative pedal curve $\mathcal{N}$ of the Reuleaux Triangle $\mathcal{R}$ w.r. to a point $M$ on its boundary consist on an 
%  point $P_0$ (the antipedal of $M$ through $V_3$) and two elliptic arcs ${A_1}{A_2}$ and ${B_1}{B_2}$ (green and blue). A sample $P$ is shown illustrating instantaneous tangency to $\mathcal{N}$ at $P'$.}
%   \label{fig:envelope}
%\end{figure}

The Reuleaux triangle $\mathcal{R}$ is  the convex curve formed by the arcs of three circles of equal radii $r$ centered on the vertices 
$V_1,V_2,V_3$ of an equilateral triangle and that mutually intercepts in these vertices; see Figure~\ref{fig:basic}. This triangle is mostly known  due to its  constant width property \cite{yaglom1961}. 
 
  %Figure~\ref{fig:basic}.

\blfootnote{
\textbf{Keywords and phrases:} conic, inversion, pole, polar, dual curve, negative pedal curve. 

\textbf{(2020)Mathematics Subject Classification: } 51P99, 60A99.

%Received: 19.08.2020. In revised form: 26.01.2021. Accepted: 08.10.2020
} 
 
Here, we study some properties of the negative pedal curve $\mathcal{N}$ of $\mathcal{R}$ w.r. to a pedal 
point $M$ lying on one of its sides. 
This curve is the envelope of lines passing through points $P$ on  $\mathcal{R}$
and perpendicular to $PM$ \cite[Negative Pedal Curve]{mw}. 
%see Figure~\ref{fig:envelope}. 
Trivially, the negative pedal curve of arc $V_1V_2$ is a point which we call $P_0$. We show that the negative pedal curves of the other two sides are elliptic arcs with a common focus on $M$ and whose major axis measures $2r$; see Prop.~\ref{prop:elipse-reciproca}.

Let $V_3$ be the center of the circular arc where pedal point $M$ lies, and let $V_1,V_2$ be the endpoints of said arc.
Let arc $A_1 A_2$ (resp. $B_1 B_2$) be the negative pedal image of the Reuleaux side  $V_1 V_3$ (resp. $V_2 V_3)$ where point $A_1$ is the image of $V_1$, and $B_1$ the image of $V_2$.  The endpoints of $\mathcal{N}$ whose preimage is $V_3$ are respectively $A_2$ when $V_3$ is regarded as a point of side ${V_1 V_3}$, and $B_2$ when $V_3$ is regarded as a point of side ${V_2 V_3}$ of the Reuleaux triangle. 

%\subsection*{Main Result}  

Our main result (Theorem~\ref{thm:endpoint}, Section~\ref{sec:endpoint-conic})
is an intriguing property of the conic $\mathcal{C}^*$ -- called here the {\em endpoint conic} -- that passes through the endpoints $A_1$, $A_2$, $B_1$, $B_2$ of the negative pedal curve $\mathcal{N}$, and through $P_0$: that one of its foci is precisely the pedal point $M$; see Figure~\ref{fig:endpoint-conic}. We also give a  full geometric  description of its axes, directrix and vertices, and a criterion for identifying its type, according to the location of the pedal point $M$.

In Section~\ref{sec:elementary} we  prove other  properties of the
Reuleaux triangle and its negative pedal curve, involving tangencies, collinearities and homotheties.

The proofs combine elementary techniques with inversive arguments  and polar reciprocity. A review of polar reciprocity and other  concepts, 
including the description  of the negative pedal curve as a locus of points, as well as an alternative  description of it as an envelope of lines is postponed to the Appendix.

\begin{figure}[H]
    \centering
    \includegraphics[trim=250 20 400 40,clip,width=.7\textwidth]{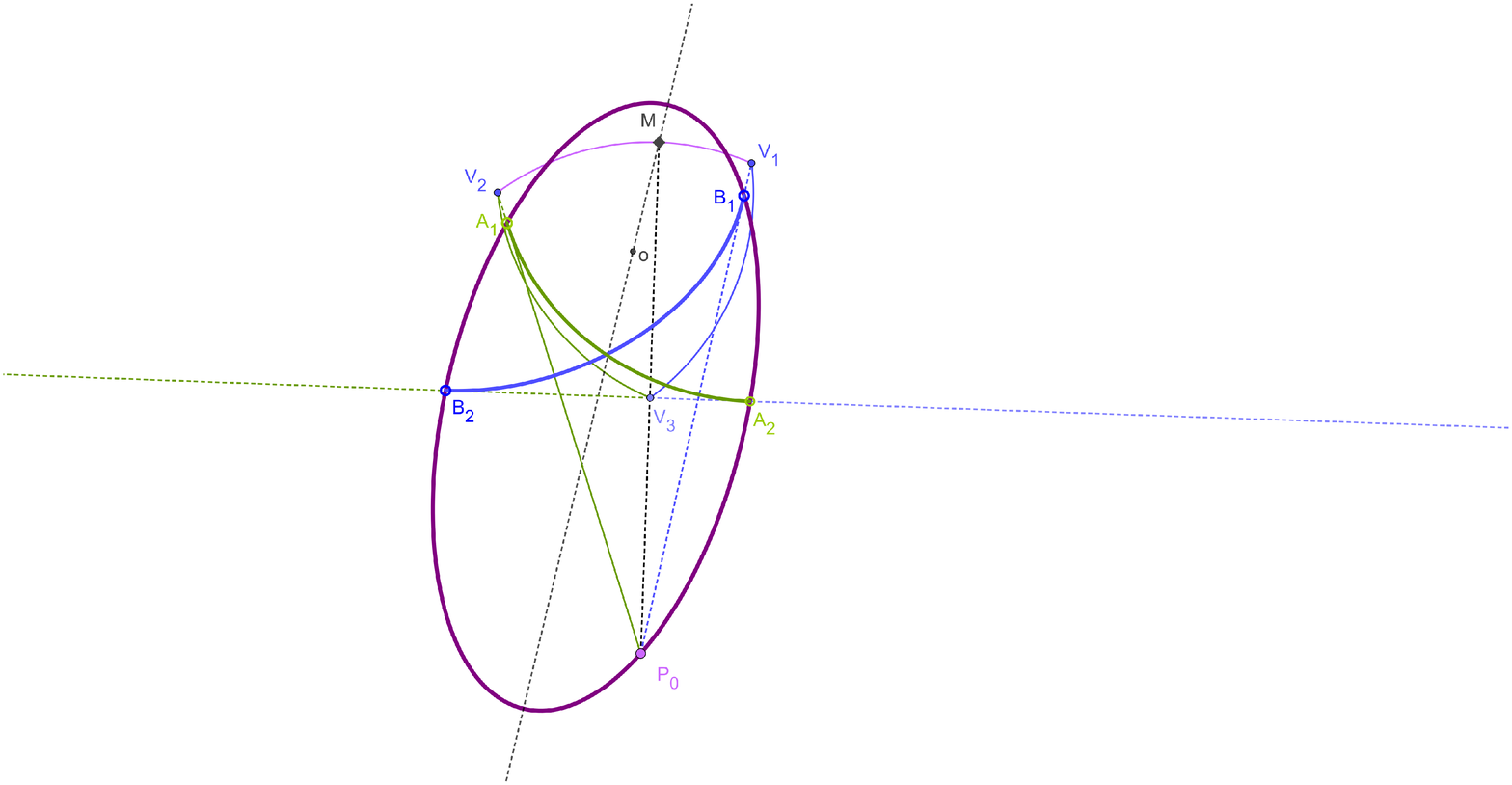}
    \caption{The sides of the Reuleaux  $\mathcal{R}$ are three circular arcs centered at the  vertices $V_1,V_2,V_3$ of an equilateral triangle. Its negative pedal curve $\mathcal{N}$  w.r. to a pedal point $M$ on its boundary consists of a point $P_0$ (the antipode of $M$ through $V_3$) and of two elliptic arcs ${A_1}{A_2}$ and ${B_1}{B_2}$ (green and blue).
    %A sample $P$ is shown %illustrating instantaneous tangency to $\mathcal{N}$ at %$P'$.
    The endpoint conic $\mathcal{C}^*$ (purple) passes through $P_0$
    and the four endpoints of the two elliptic arcs of $\mathcal{N}$. It has a
    focus on $M$ and its focal axis passes through  the center of the Reuleaux triangle. }
    \label{fig:endpoint-conic}
\end{figure}

%\subsection*{Further Results}

%(Section~\ref{sec:coll}, %Propositions~\ref{prop:coll-first}--\ref{%prop:coll-last}), and triangles and %homotheties (Section~\ref{sec:hom}, %Propositions~\ref{prop:hom-first}--\ref{p%rop:hom-last}). 

%\section{Polar Reciprocity and Negative Pedal Curves}
%\label{sec:polar}
%\input{015_polar}

\section{Main Result: The Endpoint Conic}
\label{sec:endpoint-conic}
%Our main result (Theorem~\ref{thm:endpoint}), 
%is an intriguing  propriety of the five-point conic associated to the negative pedal of a Reuleaux,
%w.r. to a pedal point $M$ located on its boundary.
% We shall call this conic,  which passes through the extremities of the negative pedal of the sides of the Reuleaux, the endpoint conic. 

\noindent Referring to Figure ~\ref{fig:endpoint-conic}.

\begin{theorem}
The conic $\mathcal{C}^*$ which contains $P_0$ and the endpoints $A_1$, $A_2$, $B_1$, $B_2$ of the negative pedal curve of the Reuleaux triangle $\mathcal{R}$ with respect to a pedal point $M$ located on a side of $\mathcal{R}$, has one focus on $M$
and its axes pass through the circumcenter of $\triangle{V_1V_2V_3}$.
\label{thm:endpoint}
\end{theorem}

The proof will require some additional steps which steadily use an inversive approach and polar duality.

The main idea is to use polar reciprocity:
in order to prove that the focus of the $C^*$ is $M$,
we show that there exists a circle $\mathcal{C}$
to which the five polars of the endpoints
$A_0,A_1,B_0,B_1$ and $P_0$
are tangent.

Recall that whenever we perform a polar transform (or a polar duality) of a conic, w.r. to an inversion circle (see Figure~\ref{fig:the-inverted-triangle}),
points on the conic transform into their polars w.r. to the inversion circle,  and those polars become tangents to the dual curve of the (initial) conic \cite[Article 306]{salmon1869}.

This indirect and inversive approach is appropriate as the polars of $A_1$, $A_2$, $B_1$, $B_2$, and $P_0$ can be readily analyzed.

Classic facts about polar reciprocity guarantee that the dual (curve) of a conic is  a circle iff the center of the inversion circle (which, in our case, is $M$) is the focus of the endpoint  conic; see Proposition~\ref{prop:circulo_dual_conica}.

The  reader not familiar with the topic may find useful the details in  Appendix and the references therein.

\begin{figure}
    \centering
    \includegraphics[trim=350 100 550 0,clip,width=1.0\textwidth]{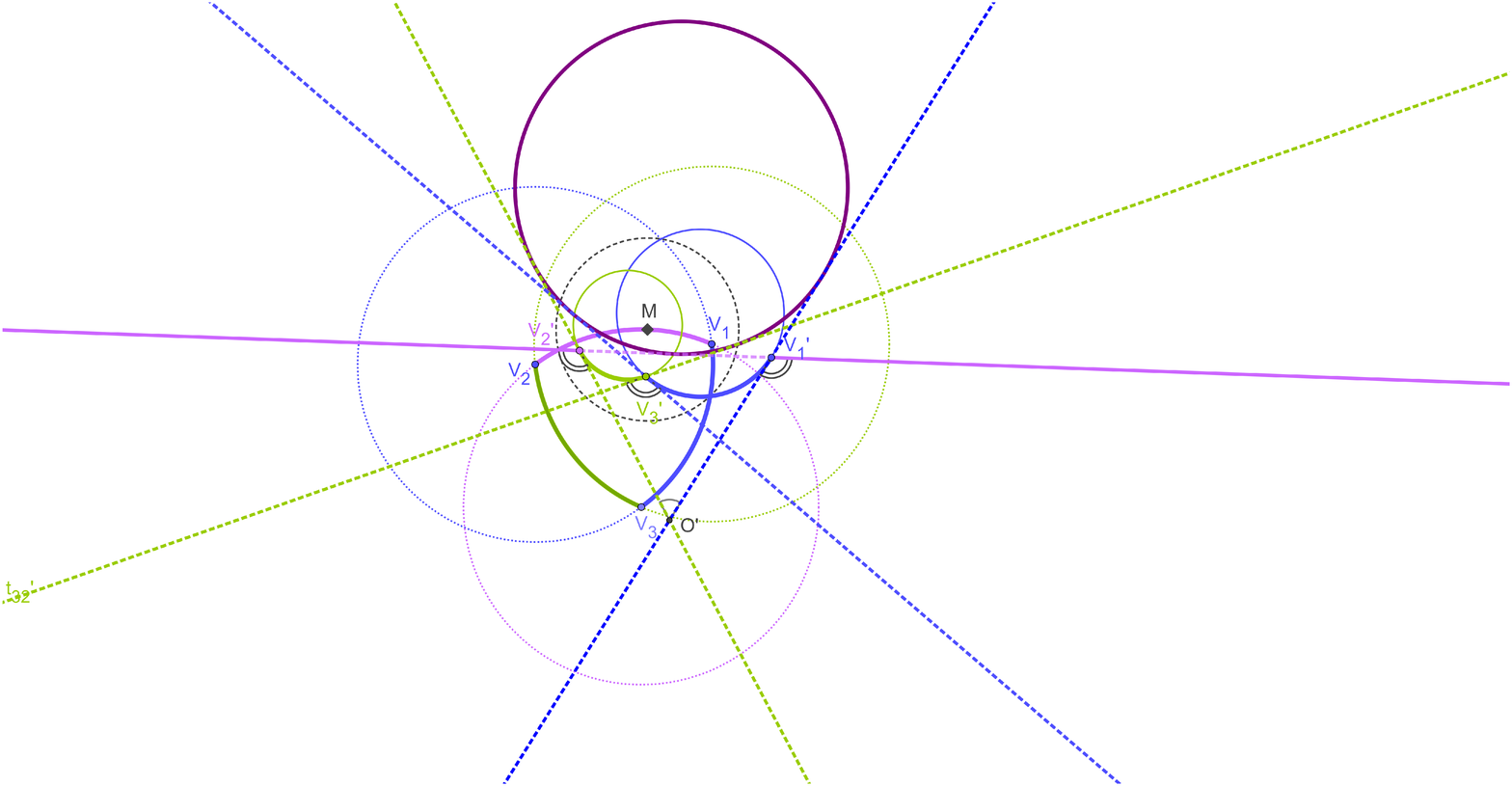}
    \caption{The Reuleaux triangle $\mathcal{R}$  and its inverse:
    (i) arcs $V_1'V_3'$ (blue) and $V_2'V_3'$ (green) are the inverses of sides $V_1V_3$ and $V_2V_3$ of $\mathcal{R}$, while line $V_1'V_2'$ except for segment
    $[V_1'V_2']$ itself (violet) is the image of arc $V_1V_2$;
    (ii) the polars of $A_1$ and $B_1$ are the tangents 
    in $V_1'$ and $V_2'$ to arcs $V_1'V_3'$ and $V_2'V_3'$; 
    (iii) the polars of $A_2$ and $B_2$  are the tangents 
    in $V_3'$ to arcs $V_1'V_3'$ and $V_2'V_3'$; 
   (iv) line $V_1'V_2'$ is the polar of $P_0$.
   (v) all five polars above are tangent to the exinscribed circle $\bf{c}$ of $\triangle{V_1'O'V_2'}$ (purple).
   The angle between circles $c_0$ (green) and $C_0$ (blue) in $V_3'$ is $120^{\circ}$ iff $V_3'$ is on the
     exinscribed $\bf{c}$ centered in $O$ (not shown); in this case, the
   tangents at $V_3'$ (dashed green and blue) to  circles
$\bf{c_0}$ and $\bf{C_0}$, are also tangent to the exinscribed circle $\bf{c}$.}
    \label{fig:the-inverted-triangle}
\end{figure}
 Referring to Figure~\ref{fig:the-inverted-triangle}.
\begin{lemma}
Let $M$ be a point on the side $V_1V_2$ of the Reuleaux triangle and $\mathcal{I}$ the inversion circle.
Let $V_1',V_2',V_3'$ be the inverses of  points $V_1,V_2,V_3$ and let arcs
$V_1'V_3'$ and $V_2'V_3'$ be, respectively, the inverses of arcs $V_1V_3$
and $V_2V_3$ of the Reuleaux triangle. Then:

\begin{enumerate}
\item The polars of $A_1$ and $B_1$ are 
the tangents at $V_1'$ and $V_2'$
to circular arcs $V_1'V_3'$ and $V_2'V_3'$. The poles of the tangents at $V_1'$ and $V_2'$ 
to arcs $V_1'V_3'$ and $V_2'V_3'$, 
 are the points $A_1$, $B_1$.

\item The polars of points $A_2$, $B_2$
are the tangents
in $V_3'$ to arcs $V_1'V_3'$ and $V_2'V_3'$, respectively. The poles of the tangents at $V_3'$ 
to arcs $V_1'V_3'$ and $V_2'V_3'$, 
 are  the points $A_2$, $B_2$.

\item The inverse of arc $V_1V_2$ is 
the line $V_1'V_2'$ excluding segment $[V_1'\;V_2']$
and the polar of $P_0$ is the line $V_1'V_2'.$
\end{enumerate}
\label{lemma:inversao}
\end{lemma}
 
 %NOVO 20 de julho
%lema dos arcos invertidos

%\begin{figure}
 %   \centering
  %  \includegraphics[trim=200 50 200 %50,clip,width=.9\textwidth]{pics/0122_triangulo_reuleaux_invertido.eps}
   % \caption{The Reuleaux triangle, its inverse and some tangents}
    %\label{fig:the inverted triangle}
%\end{figure}

%\begin{figure}
 %   \centering
 %   \includegraphics[trim=100 0 100 0,clip, %width=1.5\textwidth]{pics/0121_triangulo_reuleaux_invertido.eps}
 %   \caption{The Reuleaux triangle, its inverse and some tangents}
 %   \label{fig:the inverted triangle}
%\end{figure}

\noindent Referring to  Figure~\ref{fig:the-inverted-triangle}:

\begin{proof}
Inversion w.r. to a circle maps circles to either circles or lines: thus,
 the inverses of  arcs $V_1V_3$ and 
 $V_2V_3$ are the two circular arcs $V_1'V_3'$ and $V_2'V_3'$, respectively.
On the other hand, 
since arc $V_1V_2$
 passes through the inversion center, its
 image  is the union of two half-lines.

 All other  statements derive 
from the description of the negative pedal curve  as the dual of its inverse,
 as  shown in Proposition~\ref{prop:negative-pedal-point-polar}.
\end{proof}
  
\begin{lemma}
Using the notation in Lemma~\ref{lemma:inversao}:
\begin{enumerate}
\item the angles at $V_1'$, $V_2'$, and $V_3'$ between arcs $V_1'V_2'$, $V_1'V_3'$, and $V_1'V_2'$ respectively (the inverses of the sides of the Reuleaux triangle), are $120^{\circ}$.
 
\item $\triangle V_1'O'V_2'$ determined by the tangents at $V_1'$, $V_2'$ to said arcs
and the line $V_1'V_2'$ is equilateral.
\end{enumerate}
\end{lemma}

\begin{proof}
These statements derive from the fact that inversion is preserves angles between curves.
\end{proof}
 
Given the above results, Theorem~\ref{thm:endpoint} is equivalent to the following Lemma; see Figure~\ref{fig:the-inverted-triangle}:

\begin{lemma}
 The five polars of endpoints $A_1$, $A_2$, $B_1$, $B_2$, and $P_0$ are tangent to circle $\bf{c}$,
the exinscribed circle in $\triangle V_1'O'V_2'$, which is externally-tangent to side $V_1'V_2'$.
\label{prop:conversion}
\end{lemma}

This lemma will be equivalent to the assertion that the focus of $C^*$ coincides with pedal point $M$ when we show that the two tangents at $V_3'$ to arcs $V_1'V_3'$ and $V_2'V_3'$ respectively are also tangent to the excircle of $\triangle V_1'O'V_2'$. Referring to Figure~\ref{fig:the-inverted-triangle}, this can be restated as follows:

%\begin{figure}
  %  \centering
   % \includegraphics[trim=850 100 1200 500,clip, %width=0.9\textwidth]{pics/0132_triangulo_reuleaux_lema_das_tangentes.eps}
%    \caption{The angle between circles $c_0$ (green) and $C_0$ (blue) in $V_3'$ %is $120^{\circ}$ iff $V_3'$ is on the
 %    circle centered in $O$ passing through $V_1'$ and $V_2'$; in this case the
  % tangents at $V_3'$ (dashed green and blue) to   circles
%$\bf{c_0}$ and $\bf{C_0}$, are also tangent to the 
%exinscribed circle $\bf{c}$ (purple) of $\triangle{V_1'O'V_2'}$.}
 %   \label{fig:keylema}
%\end{figure}

\begin{lemma}
Let $\triangle V_1'O'V_2'$ be equilateral and let $O$ be the center of its exinscribed circle $\bf{c}$, externally-tangent to side $[V_1'V_2']$. Let $\bf{C}$ be another  circle, concentric with $\bf{c}$,
that passes through $V_1'$ (and $V_2'$). Finally, let $\bf{c_0}$ and 
$\bf{C_0}$ be the two circles 
tangent to the sides $[OV_1']$ and $[OV_2']$ of the triangle, at $V_1'$ and $V_2'$, respectively. Then:

\begin{enumerate}
\item $\bf{c_0}$ and $\bf{C_0}$  
intersect at an angle of $120^{\circ}$ iff 
the three circles $\bf{c_0}$, $\bf{C_0}$, and $\bf{C} $ pass through one common point; let $V_3'$ be that point.
 
\item if the condition above is fulfilled, then the two tangents at the common point $V_3'$ to circles $\bf{c_0}$ and $\bf{C_0}$  are also tangent to the exinscribed circle $\bf{c}$.
\end{enumerate}
\label{lem:scalene}
\end{lemma}

%Figure~\ref{fig:keylema}.
While assertion (i) is automatic, once we identify circles $\bf{c_0}$ and $\bf{C_0}$ as the inverses of the circles which define the Reuleaux triangle,
assertion (ii) is not obvious and will require additional steps.

Though our construction is in general asymmetric, there are two regular hexagons associated with it, used in the results below.

%We shall give a symmetric proof to this apparently scalene result, based on a rudimentary (yet useful) form of Poncelet's porism for hexagons. 

%NOVO: 09-08: COMPRESSAO DOS LEMAS 

\begin{lemma}
Let $[A_0A_1{\dots}A_5]$ be a regular hexagon with inscribed circle $\bf{c}$ and circumcircle $\bf{C}$.
Let $P_0$ be a point on arc $A_0A_1$ of $\bf{C}$ and let $P_0P_1$, $P_1P_2,\dots,P_5P_6$ be the tangents from
$P_0,P_1,\dots ,P_5$ to $\bf{c}$. 

Let $\bf{c_0}$ be the circle tangent to side $[A_0A_5]$ of the hexagon at $A_0$ whose center is inside the hexagon; let $P_0$ be its second intersection point with $\bf{C}$. Then:
 
\begin{enumerate}
    \item Points $P_6$ and $P_0$ coincide and the hexagon $[P_0P_1\dots P_5]$ is regular and 
congruent with $[A_0A_1\dots A_5]$. 
Both hexagons share the same incircle and circumcircle.
\item Let  $P_0P_1$ be the tangent  from $P_0$ to  $\bf{c};$ then it   tangents (in $P_0$) the circle $\bf{c_0}$, as well.
\label{lem:tangency}
\end{enumerate}
\end{lemma}

Referring to Figure~\ref{fig:the-key-lemma-hexagono}: %\textcolor{red}{liliana}.

\begin{proof}
i) When we perform the construction of  tangent lines $P_0 P_1,{\dots},P_5 P_6$, the process ends in six steps and $P_0=P_6$, thanks to Poncelet's porism, since $\bf{c}$ and $\bf{C}$ are the incircle and the circumcircle of a hexagon. Note the latter is regular since its inscribed and circumscribed circles are concentric.
%minha vaidade!

%esta é muito mais dificil!

ii) Let $T_0$ be the intersection of the perpendicular bisector of segment $[A_0P_0]$ with line $A_0A_5$. We shall prove that $T_0P_0$ is tangent at $P_0$ to circle 
$\bf{c_0}$ and that  $T_0,P_0,P_1$ are collinear.

Let $c_0$ be the center of circle $\bf{c_0}$;
since $T_0$ is a point on the perpendicular bisector of 
$[A_0P_0]$, and since $A_0$ and $P_0$ are the two intersections of circles $\bf{c_0}$ and $\bf{c}$, 
then $O,c_0$ and $T_0$ are collinear.

Next, $\triangle{T_0A_0c_0}=\triangle{T_0P_0c_0}$, as they have respectively-congruent sides, hence:

\[ \angle{T_0P_0c_0}=\angle{T_0A_0c_0}=90^{\circ} \]

which proves that line $T_0P_0$ is tangent at $P_0$ to circle $\bf{c_0}$.

Furthermore, $\triangle{T_0A_0O}=\triangle {T_0P_0 O}$
as they have respectively-congruent sides, hence:

\[ \angle{T_0P_0O}=\angle{T_0A_0O} \]
 
By hypothesis, $T_0$, $A_0$, and $A_5$ are collinear, and $\triangle A_0A_5O$ is equilateral, hence the external angle $\angle{T_0A_0O}=120^{\circ}$; hence $\angle{T_0P_0O}=120^{\circ}$ as well. Since by (i) $\triangle P_0OP_1$ is equilateral, then $\angle{OP_0P_1}=60^{\circ}$ and  $\angle{T_0P_0P_1}=180^{\circ}$, proving that points $T_0,P_0,P_1$ are collinear.

%\textcolor{red}{colocar o lema dos circulos homoteticos}
\end{proof}

\begin{figure}[H]
    \centering
    \includegraphics[trim=190 25 230 35,clip, width=1.0\textwidth]{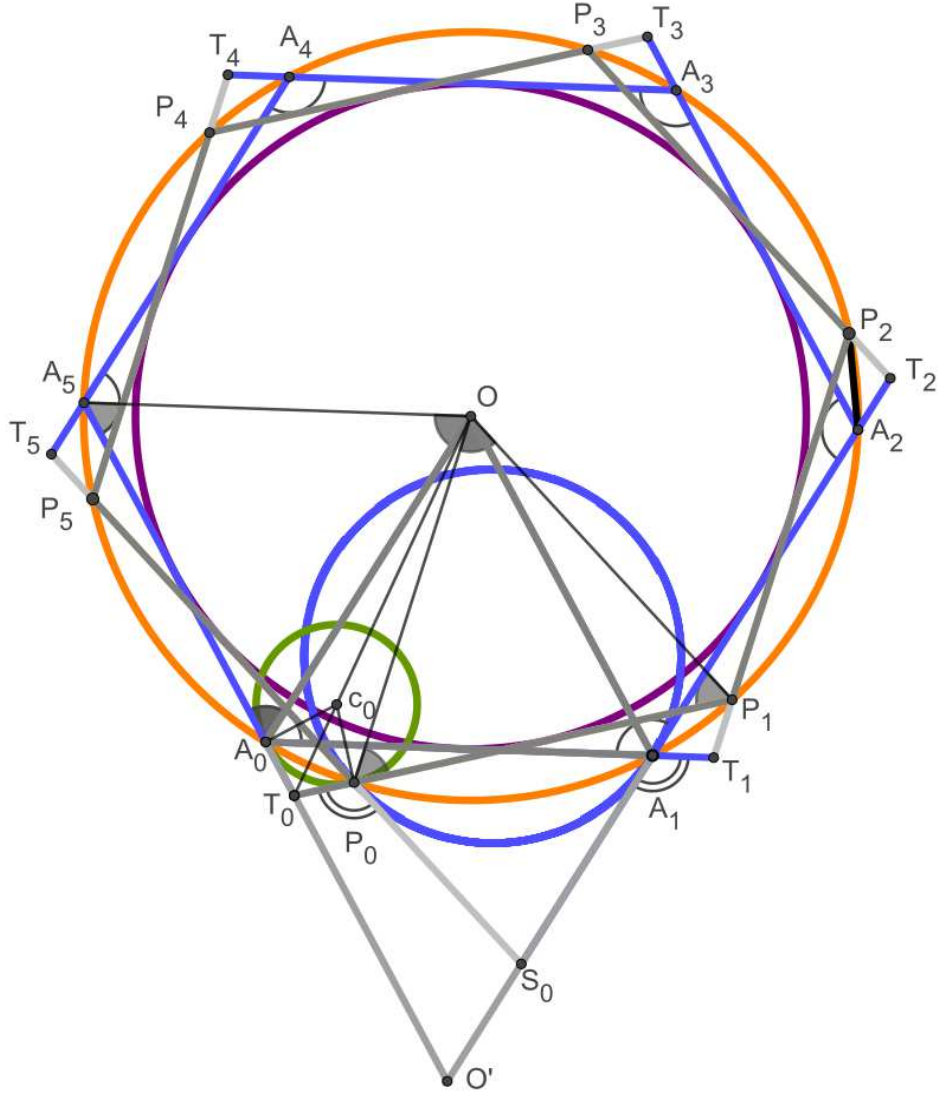}
    \caption{The angle between circles $c_0$ (green circle) and $C_0$ (blue circle) in $P_0$ is $120^{\circ}$ iff $P_0$ is on the circumcircle
   of the regular hexagon $[A_0A_1\dots A_5]$ (orange circle). }
    \label{fig:the-key-lemma-hexagono}
\end{figure}

% \label{lem:nome-do-lema}
% \label{prop:nome-da-prop}
% \label{thm:nome-do-thm}

We now reformulate  Lemma~\ref{lem:scalene} as follows:

\begin{lemma}
Let $[A_0A_1\dots A_5]$ be a regular hexagon whose incircle is $\bf{c}$ and circumcircle is $\bf{C}$. Let $\bf{c_0}$ be the circle tangent to side $[A_0A_5]$ of the hexagon at $A_0$ and let  $\bf{C_0}$ be the circle tangent to side $[A_1A_2]$ at $A_1$. Then the angle between circles $\bf{c_0}$ and $\bf{C_0}$ is $120^{\circ}$ iff the three circles $\bf{c_0},\bf{C_0}$, and $\bf{C}$ have one common point.
\label{lem:key}
\end{lemma}

\begin{proof}
$\Leftarrow$ First assume circles $\bf{c_0}$ and $\bf {C_0}$ intersect at a point $P_0$ on circumcircle $\bf {C}$. Referring to Figure~\ref{fig:the-key-lemma-hexagono}, if $P_0$ is on arc $A_0A_1$ of $\bf{C}$ then by 
Lemma~\ref{lem:tangency} $P_0P_1$ is a common tangent to circles $\bf{c_0}$ and $\bf {c}$. In particular, $P_0P_1$ is the tangent at $P_0$ to circle $\bf{c_0}$.

Next, let $P_0P_5'$ be the tangent from $P_0$ to the incircle $\bf{c}$ (distinct from $P_0P_1$). Similarly, let $P_5'P_4'$, $P_4'P_3'$, ${\dots}P_1'P_0'$ be the tangents from  $P_5',P_4', \dots, P_1'\in  \bf{C}$ to the incircle $\bf{c}$.

Then, as above, points $P_0'$ and $P_0$ coincide, and hexagon $[P_0'P_1'\dots P_5']$ is regular; since  hexagons
$[P_0P_1\dots P_5]$ and  $[P_0'P_1'\dots P_5']$ have one common point ($P_0$), are regular, and are both inscribed in $\bf {C}$, they must coincide.

Once again, Lemma~\ref{lem:tangency} guarantees that
 $P_0P_5$ is a common tangent to circles $\bf{C_0}$ and $\bf {c}$. In particular, line  $P_0P_5$ is the tangent at  $P_0$ to circle $\bf{C_0}$.
 
Since hexagon $[P_0P_1\dots P_5]$ is regular, $\angle{P_1P_0P_5}=120^{\circ}$. This guarantees that the angle between circles $\bf{c_0}$ and $\bf{C_0}$, which is the angle between their tangents at $P_0$, is also $120^{\circ}$.

%this is the tricky part
$\Rightarrow$
By hypothesis, circles $\bf{c_0}$ and $\bf{C_0}$ 
intersect at an angle of $120^{\circ}$. We shall prove that, necessarily, point $P_0$ must be on the circumcircle $\bf{C}$.

Call $P_0'$ the intersection point between circle $\bf{c_0}$ and arc $A_0A_1$ of circle $\bf{C}$. We shall  prove that $P_0$ and $P_0'$ coincide.

Let $\bf{C_0'}$ be the circle tangent at $A_1$ to line $A_1A_2$ that passes through $P_0'$. Then, by the first part of this proof, circles $\bf{c_0}$ and $\bf{C'_0}$ intersect at an angle of $120^{\circ}$. So circles ${\bf C_0}$ and ${\bf C'_0}$ would be two circles, 
both tangent at $A_1$ to line $A_1A_2$, which intersect circle ${\bf c_0}$ at the same angle. Hence circles $\bf{C_0}$ and $\bf{C'_0}$ must coincide, as do points $P_0$ and $P_0'$.
\end{proof}

Finally we can prove Theorem~\ref{thm:endpoint}:

\begin{proof}
The above lemmas prove that the focus of $\mathcal{C}^*$ coincides with $M$. We end the proof by showing that the axis of $\mathcal{C}^*$ passes through the circumcenter of $\triangle{V_1V_2V_3}$. Equivalently,
we prove that the directrix of $\mathcal{C}^*$ is perpendicular to the line
that joins points $M$ and $G$. We shall an inversive argument.

As shown in Proposition~\ref{prop:curva_dual_de_um_circulo} in the Appendix, the directrix of a conic whose polar-dual is some circle,
is precisely the polar of the center of that circle 
(w.r. to the inversion circle). In other words, the directrix of the $\mathcal{C}^*$ is the polar of $O$.

Recall $V_1'$, $V_2'$ and $V_3'$ are, respectively, the inverses of $V_1$, $V_2$, $V_3$ w.r. to the inversion circle centered in $M$. Recall also that $M$, $O$, and $G$ are, respectively, the center of the inversion circle, 
the circumcenter of $\triangle V_1'V_2'V_3'$, and 
the circumcenter of $\triangle V_1V_2V_3$. Hence $M,O,G$ are collinear.

In turn, this implies that the polar of $O$ (that is
perpendicular to $OM$), is  perpendicular
to $GM$, as well.

Thus, the axis of $\mathcal{C}^*$ and line $MO$ 
will either be parallel or coincide. Since $M$ is the focus of $\mathcal{C}^*$,  is major axis is line $MO$ and point $G$ is on that line.
\end{proof}

\begin{figure}[H]
    \centering
    \includegraphics[trim=250 30 520 0,clip,width=\textwidth]{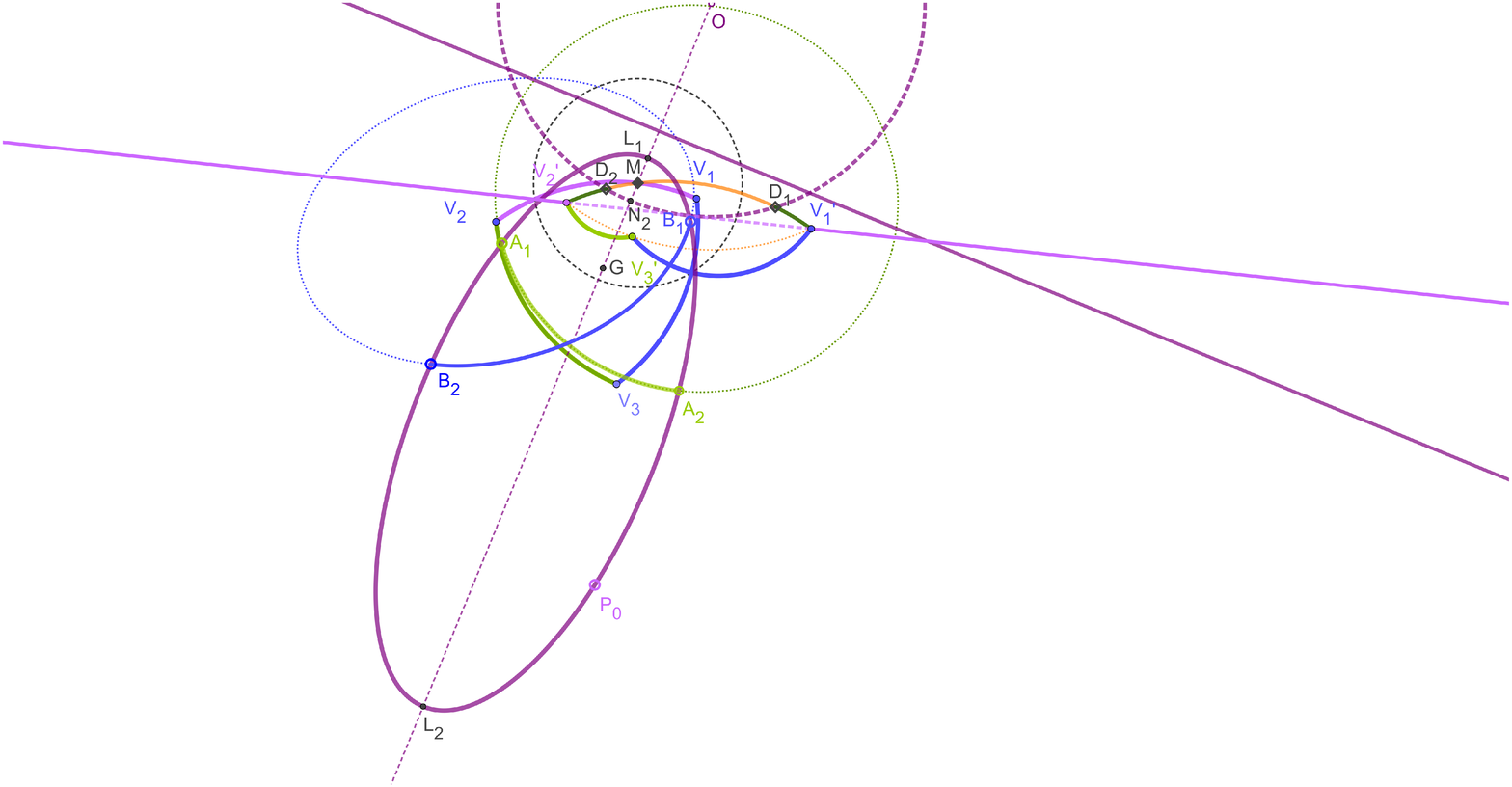}
    \caption{The endpoint conic $\mathcal{C}^*$ (purple) is the dual of circle $\bf{c}$ (dashed purple) w.r. to the inversion circle
    $\mathcal{I}$  centered on $M$ (dashed black);
     Its directrix is the polar of $O$, the circumcenter of  $\bf{c}$;
    its major axis passes through $G$, and
    its vertices, $L_1$ and $L_2$ 
    are the inverses of antipodal points $N_1$ and $N_2$, 
    the diameter of $\bf{c}$ passing through $M$ (point $N_1$, the antipode of $N_2$ w.r. to $O$, not shown).}
    \label{fig:conicafucsia}
\end{figure}

The above results reveal that the endpoint conic is in fact the polar-dual of a special circle, which depends on the vertices of the Reuleaux triangle and on the location of pedal point  $M$; see Proposition ~\ref{prop:curva_dual_de_um_circulo}.

%~\ref{fig:conicafucsia}

% \textcolor{red}{liliana}

\noindent Using the notation in Lemma~\ref{prop:conversion}:

\begin{corollary}
The endpoint conic $\mathcal{C}^*$ associated with a Reuleaux triangle and a pedal point $M$ is the polar-dual of circle  $\bf{c}$.
Its type depends on the location of $M$ on arc $V_1V_2$:
it is an ellipse (resp. hyperbola) if it lies inside (resp. outside) circle $\bf{c}$. If it is on said circle, $\mathcal{C}^*$ is a parabola.
\end{corollary} 

Therefore one can (geometrically) construct any of its elements (vertices, other focus) as well as compute its axis and eccentricity.

%The result above enables to explain  (yet not to predict) the type of the endpoint conic.

\begin{observation} 
The type of $\mathcal{C}^*$ can be identified with an additional observation. Referring to Figure~\ref{fig:paraboladetalhe}. With  the notation in lemma \ref{lem:key}:

Let $\bf{C}'$ be reflection of $\bf{C}$ w.r. to line $V_1'V_2'$ and let 
 $D_1,D_2$ be the two intersections between circles $\bf{c}$ and $\bf{C}'$. One can check that $V_3'$ is the reflection of $M$ w.r. to $V_1'V_2'$;
  therefore, $M$ is located both on the reflection of the circumcircle of $\triangle{V_1'V_2'V_3'}$ w.r. to $V_1'V_2'$, and on arc $V_1V_2$ of the Reuleaux triangle. The location of $M$ with respect to  $\bf{C}'$ reveals the which type of conic $\mathcal{C}^*$ is: it  is an ellipses if $M$ is on the arc $D_1D_2$ of circle $\bf{C}'$, a hyperbola, if $M\in V_2'D_2$ or $V_1'D_1$ and a parabola when $M$ is either $D_1$ or $D_2$.
\end{observation}

\begin{figure}[H]
    \centering
    \includegraphics[trim=150 50 150 100,clip,width=1.1\textwidth]{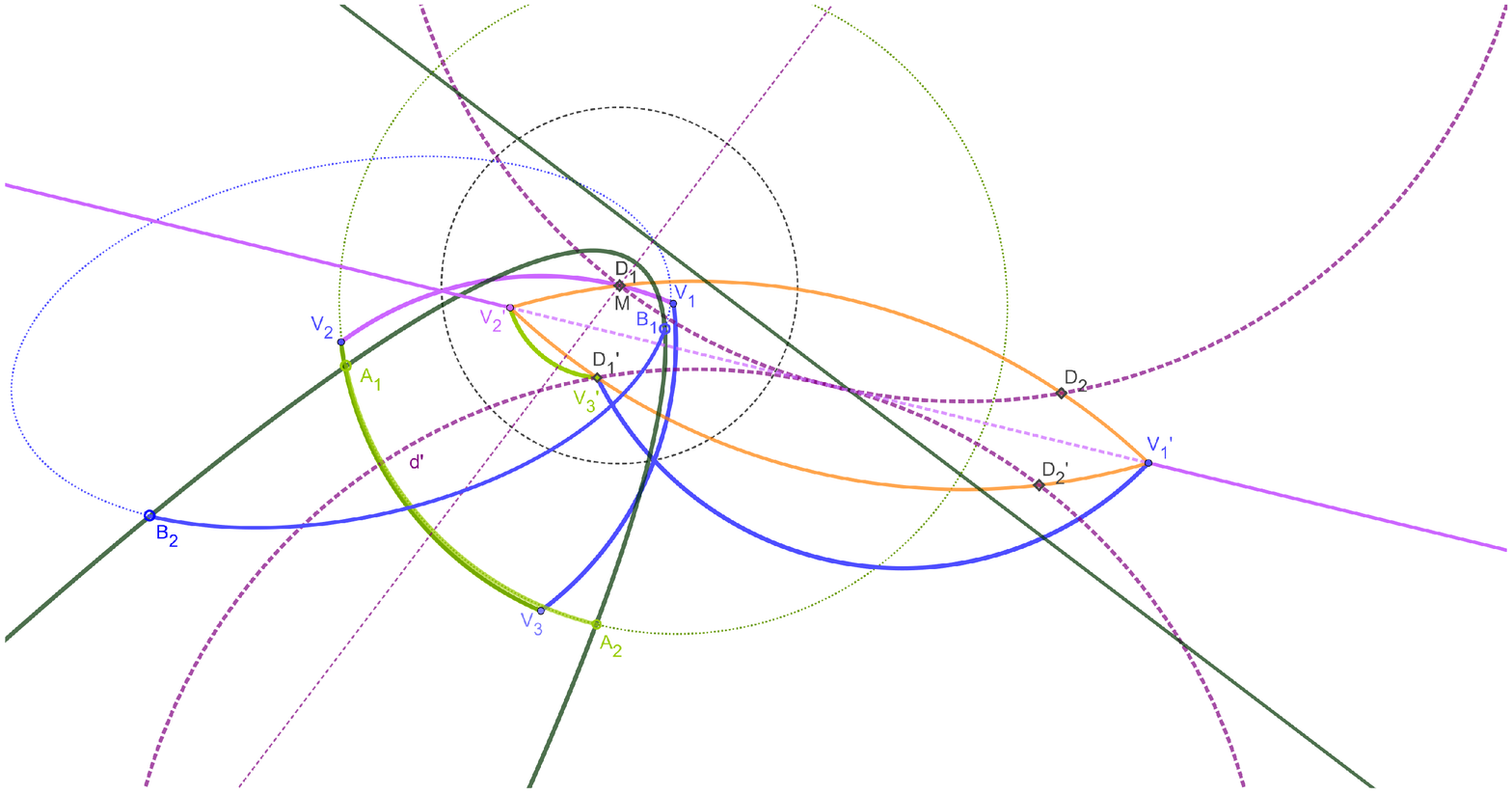}
    \caption{
    The endpoint conic $\mathcal{C}^*$ (dark green) 
    is a parabola iff point $M$ coincides with either $D_1$ or $D_2$; its directrix (dark green) is the polar of the circumcenter $O$ of $\triangle{V_1'V_2'V_3'}$ (not shown).}
    \label{fig:paraboladetalhe}
\end{figure}

%O DO:
%The last ingredient is in fact a convenient reformulation  of Lemma~\ref{lem:scalene}; its proof ends the proof of the referred lemma, hence the proof of Lemma~\ref{prop:conversion}.

\section{Some Elementary Properties}
\label{sec:elementary}
\subsection{Collinearity and Tangencies}
\label{sec:coll}

Referring to  
 Figure~\ref{fig:m-foci-circ-tri}:
\begin{proposition} The negative pedal curve  $\mathcal{N}$ of the Reuleaux triangle consists of 
 two elliptic arcs $\mathcal{E}_A$ and $\mathcal{E}_B$  and 
 a point $P_0$, the antipode of $M$ w.r. to the center of the circle where $M$ is located.
 The two ellipses 
$\mathcal{E}_A,\mathcal{E}_B$ are centered on the vertices of the Reuleaux triangle, $V_1$ and $V_2$,  have one common focus at $M$, and their semi-axes are of length  $r$.
\label{prop:elipse-reciproca}
\end{proposition}

\begin{proof}
By hypothesis, $M$ belongs to arc ${V_1}{V_2}$ of the circle centered in $V_3$ that passes through $V_2$ and $V_3$. Hence, if $P$ is any point on this arc and we draw the perpendicular $p$ through $P$ on $PM$, all these lines will pass through a fixed point $P_0$, which is  he antipode  of $M$ w.r. to center $V_3$.

The second part derives directly from the general construction of the negative pedal curve of a circle. See Proposition~\ref{prop:negative-pedal-point-polar}
 in the Appendix.
\end{proof}
\label{prop:coll-first}

%\begin{figure}
 %   \centering
 % \includegraphics[trim=230 10 320 %10,clip,width=0.9\textwidth]{pics/0051_branch_ellips_circs.eps}
 %   \caption{The two branches of the $\mathcal{N}$ are arcs of ellipses %$\mathcal{E}_A$ and $\mathcal{E}_B$ (green and blue),  centered on Reuleaux %vertices $V_1,V_2$, respectively. They have a common focus in $M,$ and  the %other foci are  $f_A,f_B$. The lenghts of their main axes is $2r,$ the same as %the diameters of the three Reuleaux circles (dashed). $P_0,A_1,V_2.$ Points 
  %  $P_0,B_1,V_1$ are collinear and along their minor axis. $P_0A_1$ and $P_0B_1$ %are tangent to $\mathcal{E}_A$ and $\mathcal{E}_B$, respectively. $A_2B_2$ is %tangent to both ellipses and $A_2,B_2,V3$ are collinear.}
  %  \label{fig:branch-ells-w-circs}
%\end{figure}

\begin{proposition}
The minor axes of  ${\mathcal{E_A}}$ and ${\mathcal{E_B}}$ pass through $P_0$.
\end{proposition}

\begin{proof}
By the definition of the negative pedal curve, if we regard $V_1$ as a point on arc ${V_1}{V_2}$ of the circle centered on $V_3$ on which the pedal point 
$M$ lies, then $P_0V_1$ will be perpendicular to $MV_1$. Since $V_1$ is the center of $\mathcal{E_A}$ and since  line $MV_1$ is its major axis, its minor axis will be along $P_0V_1$. Similarly, the minor axis of $\mathcal{E_B}$ is $P_0V_2$.
\end{proof}

\begin{proposition}
 Points $A_2$, $B_2$ and $V_3$ are collinear and line $A_2B_2$ is a common tangent to $\mathcal{E}_A$ and $\mathcal{E}_B$.
  \label{prop:hom-first}
\end{proposition}

\begin{proof}
By construction, the negative pedal curve of arc $V_2V_3$ is the elliptic arc $\mathcal{E_A}$, delimited by  $A_1$ and $A_2$. This implies that  $M V_3$ and $A_2V_3$ are perpendicular, as well as  $M V_3$ and $B_2V_3$. Thus points $A_2$, $V_3$ and $B_2$ are collinear. Also by construction,
the perpendicular to ${M}{V_3}$ at $V_3$ is tangent to $\mathcal{N}$ at $A_2$ (resp. $B_2$) when $V_3$ is regarded as a point in the 
$V_2V_3$ (resp. $V_1V_3$) arc.
Hence the points $A_2,V_3$ and $B_2$ are collinear ($\angle{A_2V_3B_2=180^{\circ}}$)
and $A_2B_2$ is the common tangent to $\mathcal{E_A}$
and $ \mathcal{E_B}$, at $A_2$ and $B_2$, respectively.
\end{proof}

\begin{proposition}
Point $A_1$ is on $P_0V_2$ and $B_1$ is on $P_0V_1$.
\end{proposition}

\begin{proof}
If we regard $V_1$ as a point on arc ${V_1}{V_2}$ of the circle centered on $V_3$ whose negative pedal curve is $P_0$, then, necessarily,
$V_1P_0\perp M V_1$. Similarly, 
if we regard $V_1$ as a point on arc $V_1V_3$ of the circle centered on $V_2$ whose negative pedal is $\mathcal{E_B}$, then by $\mathcal{N}$'s construction
$B_1V_1\perp M V_1$.
 Since this perpendicular must be unique, $P_0$, $B_1$, and $V_1$ are collinear as will be $P_0$, $A_1$, and $V_2$. 
\end{proof}

\begin{proposition}
The line joining the intersection points of $\mathcal{E}_A$ and $\mathcal{E}_B$ is the perpendicular  bisector of segment
$[f_A f_B]$ and also passes through $P_0$.
 \label{prop:coll-last}
\end{proposition}

\begin{proof}
Let $U_1,U_2$ denote the points where $\mathcal{E_A}$ and $\mathcal{E_B}$ intersect.
In order to prove that $P_0$, $U_1$, and $U_2$ are collinear, we show each one lies on the perpendicular bisector of $[f_A f_B]$. Since $U_1$ (resp. $U_2$) is on $\mathcal{E_A}$ (resp. $\mathcal{E_B}$), whose foci are $M$ and $f_A$ (resp. $M$ and $f_B$), with major axis of length $2r$, then

\[ U_1f_A+U_1M=2r;\;\;\;\;\;\; U_2f_B+U_1M=2r. \]

This implies that $U_1f_A= U_1f_B$ and $U_2f_A= U_2f_B$, hence  both $U_1$ and $U_2$  belong to the perpendicular bisector of 
 $[{f_A}{f_B}]$. Since we've already shown that ${{P_0}{V_1}}\perp{{M}{V_1}}$, and since 
 $V_1$ is the center of ${M}{f_A}$,
 this means that 
 $P_0V_1$ is the perpendicular bisector of $[Mf_A]$ and this implies that 
 $P_0f_A=P_0M$. Similarly, $P_0f_B=P_0M$, and
 hence 
 $P_0f_A= P_0f_B$. Therefore 
 $P_0$ is also on the perpendicular bisector of 
 $[f_A f_B]$, ending the proof.
\end{proof}

\begin{figure}
    \centering
    \includegraphics[trim=200 10 280 10,clip,width=1.0\textwidth]{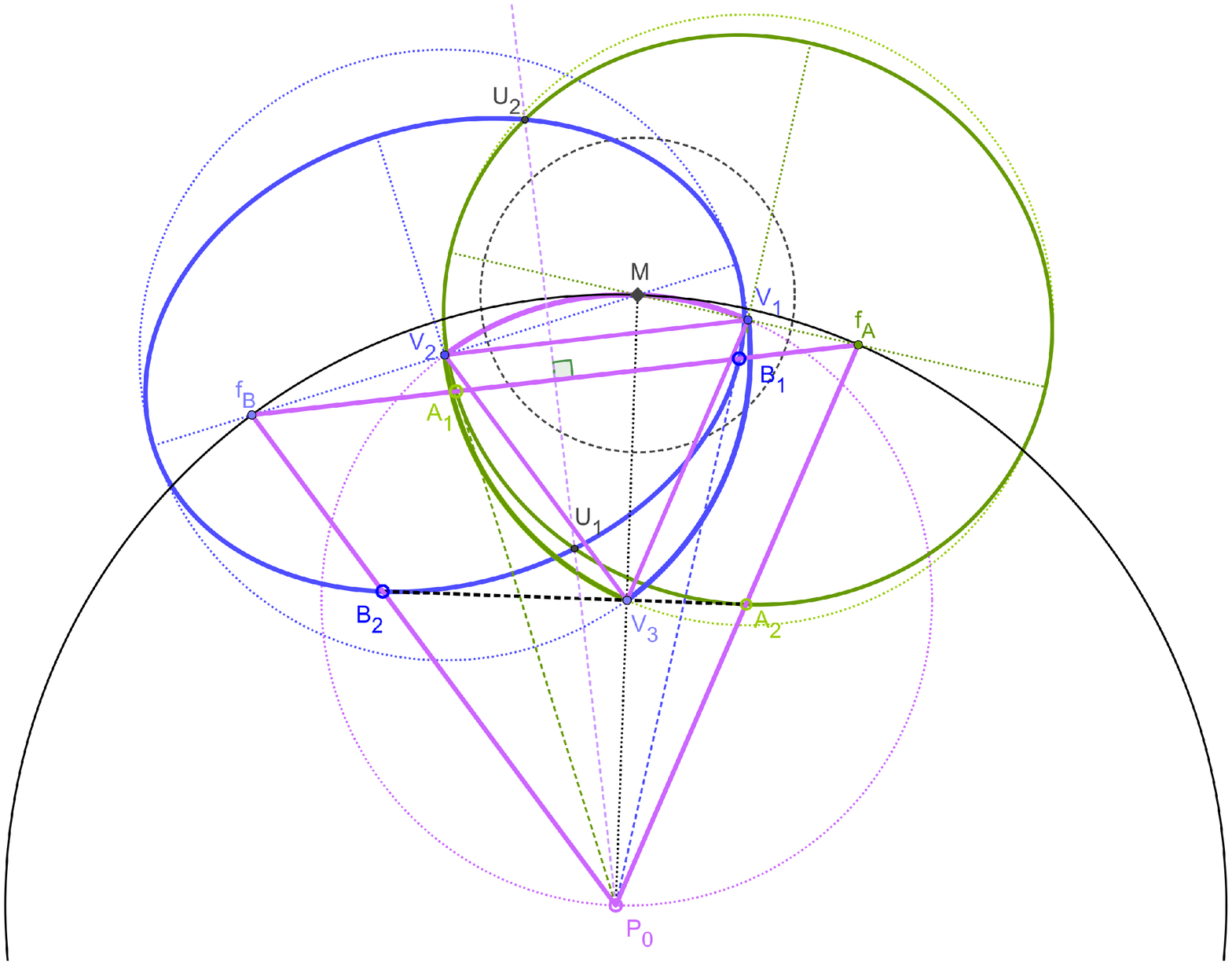}
    \caption{The negative pedal curve 
     $\mathcal{N}$ w.r. to pedal point $M$ consists of two arcs of ellipses $\mathcal{E}_A$ and $\mathcal{E}_B$ (green and blue),  centered on Reuleaux vertices $V_1,V_2$, respectively. They have a common focus at $M$, and  the other foci are  $f_A,f_B$. Their major axes have length of $2r$, equal to the diameters of the three Reuleaux circles (dashed). Points $P_0,A_1,V_2$ 
   $P_0,B_1,V_1$ are collinear
   and along their minor axis. The lines
   $P_0A_1$ and $P_0B_1$ are tangent to $\mathcal{E}_A$ and $\mathcal{E}_B$, respectively. $A_2B_2$ is tangent to both ellipses and $A_2,B_2,V3$ are collinear.
    The circle (black) passing through $M$, $f_A$ and $f_B$   $\mathcal{E}_A,\mathcal{E}_B$ (green and blue) is centered on $P_0$ (antipodal of $M$ w.r. to $V_3$). The
     distance between the foci $f_A$ and $f_B$ is constant.
    Triangle $\mathcal{T}=\triangle{f_Af_BP_0}$ is equilateral and  its sides pass through (i) $A_2$, (ii) $B_2$, (iii) $A_1,B_1$, respectively. Both intersections $U_1,U_2$ of $\mathcal{E}_A$ with $\mathcal{E}_B$ lie on the perpendicular bisector of ${f_A}{f_B}$, hence  are collinear with $P_0$.  $\mathcal{T}$ and $\triangle{V_1}{V_2}{V_3}$ are 
    homothetic  (homothety center $M$ and homothety ratio   $2)$).}
    \label{fig:m-foci-circ-tri}
\end{figure}
 
\subsection{Triangles and Homotheties}
\label{sec:hom}

\noindent Referring to Figure~\ref{fig:m-foci-circ-tri}:

\begin{proposition}
 The two sides of triangle $\triangle f_AP_0f_B$, incident on $P_0$, contain points $A_2$ and $B_2$. The other side contains points $A_1$ and $B_1$. 
\end{proposition}

\begin{proof}
The construction of the negative pedal curve of arc $V_2V_3$ implies $A_1V_2\perp MV_2$.
Since $V_2$ is the center of the $\mathcal{E}_A$, $A_1V_2$ is the perpendicular bisector of $[Mf_B]$
 hence $A_1f_B=A_1M$. Since $A_1$ lies on $\mathcal{E}_A$, $MA_1+f_A A_1=2r$, 
hence $f_B A_1+f_A A_1=f_A f_B$. Therefore,
triangle inequality implies $f_B$, $A_1$, and $f_A$ must be collinear. A similar proof applies to $B_1$. In order to prove that $P_0$, $B_2$, and $f_A$ are collinear, we simply show  that $P_0f_A=P_0A_2+A_2f_A$. As noted above, $A_2V_3$ is perpendicular to $P_0M$ and $V_3$ is its midpoint. Hence $A_2V_3$ is the perpendicular bisector of $[P_0M]$; so $P_0A_2=MA_2$. Since $A_2$ lies on ${\mathcal E_A}$, we have:

\[ P_0A_2+A_2f_A'=MA_2+A_2f_A'=2r \]

The proof for $B_2$ is similar.
\end{proof}

\begin{proposition}
  Triangles $\triangle f_A f_B P_0$ and $\triangle V_1V_2V_3$
are homothetic at ratio 2, and with $M$ the homothety center. Hence, 
  $\triangle f_A f_B P_0$ is equilateral and the distance between $f_A$ and $f_B$ is
 the same as $2r$. Furthermore, their barycenters $X_2$ and $X_2'$ are collinear with $M$.
   \label{prop:hom-last}
\end{proposition}

\begin{proof}
Points $V_1,V_2,V_3$ are the midpoints of $Mf_A$, $Mf_B$, and $P_0M$, respectively.
Thus, ${V_1}{V_2}$ is a mid-base
of $\triangle f_AMf_B,$ $V_2V_3$ is a mid-base of $\triangle{f_B P_0f_A}$ and  
$V_3V_1$ is a mid-base of $\triangle{P_0Mf_A}$.
Hence $\triangle f_A f_B P_0$ and $\triangle {V_1}{V_2}{V_3}$
are homothetic with ratio 2, and homothety center $M$. Therefore 
  $\triangle f_A f_B P_0$ is equilateral and
 the distance between $f_A$ and $f_B$ is
 the same as the diameter $2r$ of the circles that form the Reuleaux triangle.

Thus, $\triangle{f_A f_B P_0}$ is equilateral with sides twice that of the original triangle: $f_A f_B=2{V_1}{V_2}$. This shows that the distance between the pair of foci of $\mathcal{E_A}$
and $\mathcal{E_B}$ is constant and equal to the length of their major axes. Note that lines 
$V_1f_A,V_2f_B,P_0V_3$ intersect at $M$, hence the two triangles 
are perspective at $M$. Due to the parallelism of their sides, their medians 
will be respectively parallel; let $X_2$ and $X_2'$ denote the barycenters of triangles $\triangle{V_1V_2V_3}$ and $\triangle{f_Af_BP_0}$, respectively. The barycenter divides the medians in equal proportions, which guarantees $\triangle  MX_2'V_2\sim \triangle M{X_2}f_B$. Since $M,V_2,f_B$ are collinear, so are $M,X_2',X_2$.
\end{proof}

\section{Conclusion}
\label{sec:conclusion}
We studied some proprieties of the negative pedal curve to an object with a remarkable symmetry: the equilateral Reuleaux triangle. The five endpoints of this curve determine a conic with one focus on pedal point $M$, for any choice of $M$ on the third arc of the Reuleaux.

To prove that we adopted an inversive approach,
based on the fact that the reciprocal of a conic w.r. to a circle is a circle, iff the center of inversion is the focus of said conic.

Since points on the original curve convert to lines tangent to the reciprocal curve, it sufficed to show that the five lines tangent to the inverted sides of the Reuleaux, at their endpoints, are tangent to some circle. Our proof relies on Poncelet's porism. 

One may also consider an asymmetric  Reuleaux triangle delimited by three circles 
whose radii $r_1,r_2,r_3$ are distinct, and whose centers $O_i$ are not necessarily
vertices of an equilateral triangle. Preliminary experiments show that some properties of the equilateral Reuleaux still in the asymmetric case. Using the notation in Figure ~\ref{fig:m-foci-circ-tri}, one observes that for any choice of $O_1,O_2,O_3,r_1,r_2,r_3$: 

%\begin{figure}
 %   \centering
 %   \includegraphics[width=\textwidth]{pics/0070_asym_reuleaux.eps}
 %   \caption{Asymmetric Reuleaux Triangles (blue). \textbf{Left:} Reuleaux Circles (dashed %black) have equilateral centers, but different radii $r_i$. \textbf{Middle:} non-equilateral %centers, with $r_1=r_2=r_3$; \textbf{Right:} centers non-equilateral, with different radii.}
  %  \label{fig:asym-reuleaux}
%\end{figure}

\begin{enumerate}
    \item the distance between foci $|f_Af_B|$ does not depend
    on the location of $M$.
    \item ${V_2},{A_1},{P_0}$ as well as  ${V_1},{B_1},{P_0}$ are collinear.
    \item the line through the two intersections ${U_1},{U_2}$ of $\mathcal{E}_A$ with $\mathcal{E}_B$ is perpendicular to ${f_A}{f_B}$.
    \item line ${A_2}{B_2}$ is tangent to both $\mathcal{E}_A$ and $\mathcal{E}_B$.
\end{enumerate}

%\begin{figure}
 %   \centering
  %  \includegraphics[width=\textwidth]{pics/0080_asym_reuleaux_grid.eps}
   % \caption{Geometry with asymmetric reuleaux triangles. (a) $r_1=r_2$ and $r_3={\rho}r_1$, $\rho=4/5$. $M$ in the %middle of its arc, (b) same as previous, but with $M$ slid CW along its arc. (c) $r_1=r_3$, $r_2={\rho}r_1$, $M$ %in the middle of its arc, (d) same as previous, with $M$ slid CW, showing the cusp ellipse $\mathcal{E}^*$ (green) %becoming a hyperbola. Notice neither of its foci are congruent with $M$, as occurs in the symmetric case.}
%    \label{fig:asym-reuleaux-grid}
%\end{figure}

Nevertheless, in this general setting, the focus of the 
endpoint conic no longer coincides with the pedal point $M$. Below, some open questions we could not yet answer synthetically:

\begin{itemize}
    \item What is the location of the focus of the endpoint conic if the Reuleaux triangle is asymmetric? Is it still geometrically meaningful?
    \item  When does the focus of the endpoint conic of an arbitrary Reuleaux triangle coincides with the pedal point?
    \item Are there any special locations of $M$ on arc $V_1V_2$ for which it still the focus of the endpoint conic?
    \item Is there a poristic family of Reuleaux triangles whose endpoint conic has a focus on $M$ (i.e. for any choice of the pedal point on the third arc of the Reuleaux)?
    \item What are the bounds on the eccentricity of the endpoint conic associated with some Reuleaux triangle? 
\end{itemize}

%We presume one must abdicate the idea of a synthetic proof, in order to provide an  answer to  these questions.

%We very much welcome reader contributions to add to the list of proofs and/or new discoveries.

%\begin{table}[H]
%\scriptsize
%\begin{tabular}{|c|l|l|l|}
%\hline
%\href{youtu.be}{PL\#} & Title & Narrated \\
%\hline

%\href{youtu.be}{01} & 
%\makecell[lt]{xxx} & no \\
%\hline
%\end{tabular}
%\caption{Playlist of videos. Column ``PL\#'' indicates the entry %within the playlist.}
%\label{tab:videos}
%\end{table}

\appendix

\section{Duality and the Negative Pedal Curve}
\label{app:polar}
Here we review concepts and results on polar duals \cite{salmon1869}.
%~\ref{sec:endpoint-conic}.

Let a circle $\mathcal{I}$ be called the inversion circle and its center $M$ the inversion point. Assume all inversions, the poles and polars below are performed w.r. to $\mathcal{I}$.

%Below let $\Gamma$ denote a regular curve, $\Gamma'$, denote the inverse of $\Gamma$, $A'$ the inverse of a point $A$, and $\Gamma^*$ the dual (reciprocal curve) of $\Gamma$.

The following result provides two equivalent definitions for the dual curve:

\begin{theorem}
Let $\Gamma$ be a regular curve, $\Gamma_1^*$ the locus of the poles of its tangents, $\Gamma_2^*$ the envelope of the polars of its points. Then $\Gamma_1^*=\Gamma_2^*$ and simply denote it 
$\Gamma^*$.

$\Gamma^*$ is a regular curve and the polars of its points are the tangents to $\Gamma$, while the poles of the tangents of $\Gamma^*$ are points of $\Gamma$.

Further more  $[\Gamma^*]^*=\Gamma$.
\label{thm:fundamental-thm-dual-curves}
\end{theorem}

%\textcolor{red}{liliana revisar}

\begin{figure}
\centering
\includegraphics[trim=200 40 200 60,clip,width=1.0\textwidth]{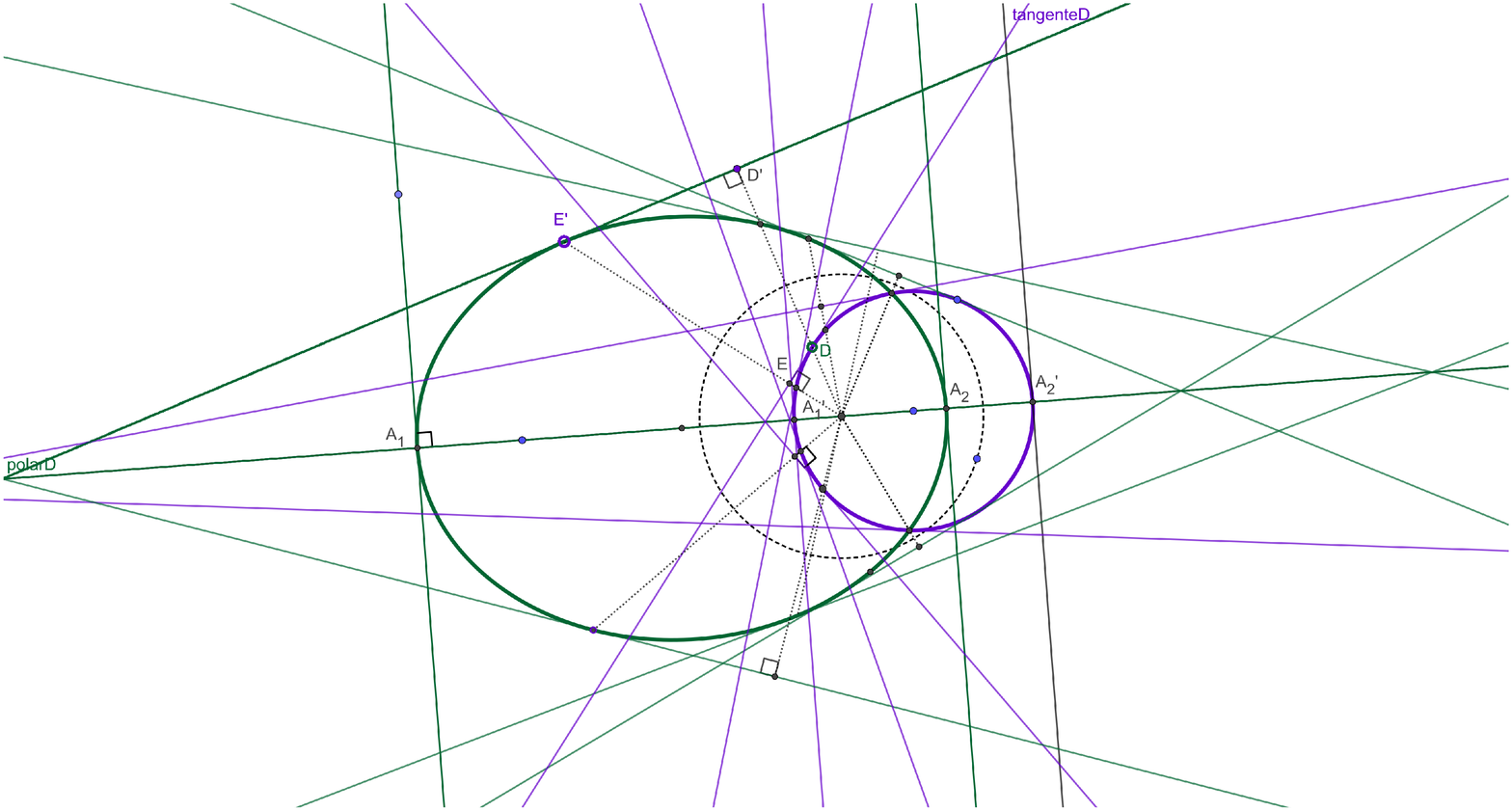}
\caption{
The dual of $\Gamma$ (violet) w.r. to its inversion circle  (dashed black) is the curve $\Gamma^*$ (dark green), the envelope of polars $E'D'$ of points $D$ on $\Gamma$, as well as loci of $E'$, the poles of the tangents $ED$ to $\Gamma$, as $D$ sweeps $\Gamma$. The dual of a circle is a conic. ${\Gamma}^*$ is an ellipse iff $M$ is inside $\gamma$, a hyperbola iff $M$ is outside $\gamma$, and a parabola, iff $M$ is on $\gamma$.  Its focus is $M$, the inversion center (not illustrated), and the directrix is the polar of its center. Its vertices are the inverses of $A'_1$ and $A'_2$, the intersection of $MO$ with $\Gamma$.}
\label{fig:ellipse-reciprocal}
\end{figure}

%\begin{figure}
 % \centering
 %   \includegraphics[trim=50 60 50 60,clip,width=1.0
 %   \textwidth]{pics/0180_reciproca_hiperbole.eps}
 %   \caption{The dual of $\gamma$ (violet) w.r. to  $\mathcal{I}$   %(dashed black) is a hyperbola $\Gamma$ (dark green) iff $M$ is outside %$\mathcal{I}$. Its focus is $M,$
  %  its main axis $[A_1A_2]$ is the inverse of the diameter $[A_1'A_2']$ %and one  directrix is the polar of $O.$ }
  %  \label{fig:hyperbola-reciprocal}
%\end{figure}

% \begin{figure}
% \centering
%    \includegraphics[trim=80 100 80 %100,clip,width=1.0\textwidth]{pics/0190_reciproca_parabola.eps}
 %   \caption{The dual of $\gamma$ (violet) w.r.  to $\mathcal {I}$ (dashed black) is a %parabola (dark green) iff $M$ is located on $\gamma$. Its focus is on $M$
  %  and its vertex is the inverse of $A_1$. The directrix is the polar of $O$.}
   % \label{fig:parabola-reciprocal}
%\end{figure}

This theorem, whose proof is based on the fundamental pole-polar theorem 
justifies the dual definition of the curve $\Gamma^*$ either as a locus of points or as an envelope of lines, and specifies who the points and the tangents at a dual curve are.

For more details on poles, polars and polar reciprocity,  see e.g. \cite{salmon1869}.

\noindent Referring to Figure~\ref{fig:ellipse-reciprocal}:
 %\ref{fig:hyperbola-reciprocal}, and \ref{fig:parabola-reciprocal}:

\begin{proposition}
 The dual (or the polar dual, or the reciprocal) 
 of a circle $\Gamma$ w.r. to an inversion circle centered at $M$
 is a conic $\Gamma^*$ whose:
 \begin{itemize}
\item focus coincides with the inversion center;
\item vertices are the inverses of the endpoints of the diameter of $\Gamma$  passing through the inversion center;
\item  directrix 
is the polar of the center of $\Gamma$.
 \end{itemize}

The dual conic $\Gamma^*$ is and ellipse (resp. hyperbola) if the center of ${\mathcal I}$ is inside (resp. outside) ${\Gamma}$ and a  parabola if said center is on said curve.
\label{prop:curva_dual_de_um_circulo}
\end{proposition}
 
\begin{proposition}
 The dual of a conic ${\Gamma}$ 
 is a circle iff the inversion center is a focus of  ${\Gamma}$.

 If this is the case, (i) the inverses of the vertices of ${\Gamma}$ are a pair of antipodal points on the dual circle $\Gamma^*$; (ii) the  center of $\Gamma^*$ is the pole of the directrix of $\Gamma.$
\label{prop:circulo_dual_conica}
\end{proposition}

\begin{figure}
    \centering
    \includegraphics[trim=110 100 60 60,clip,width=1.0\textwidth]{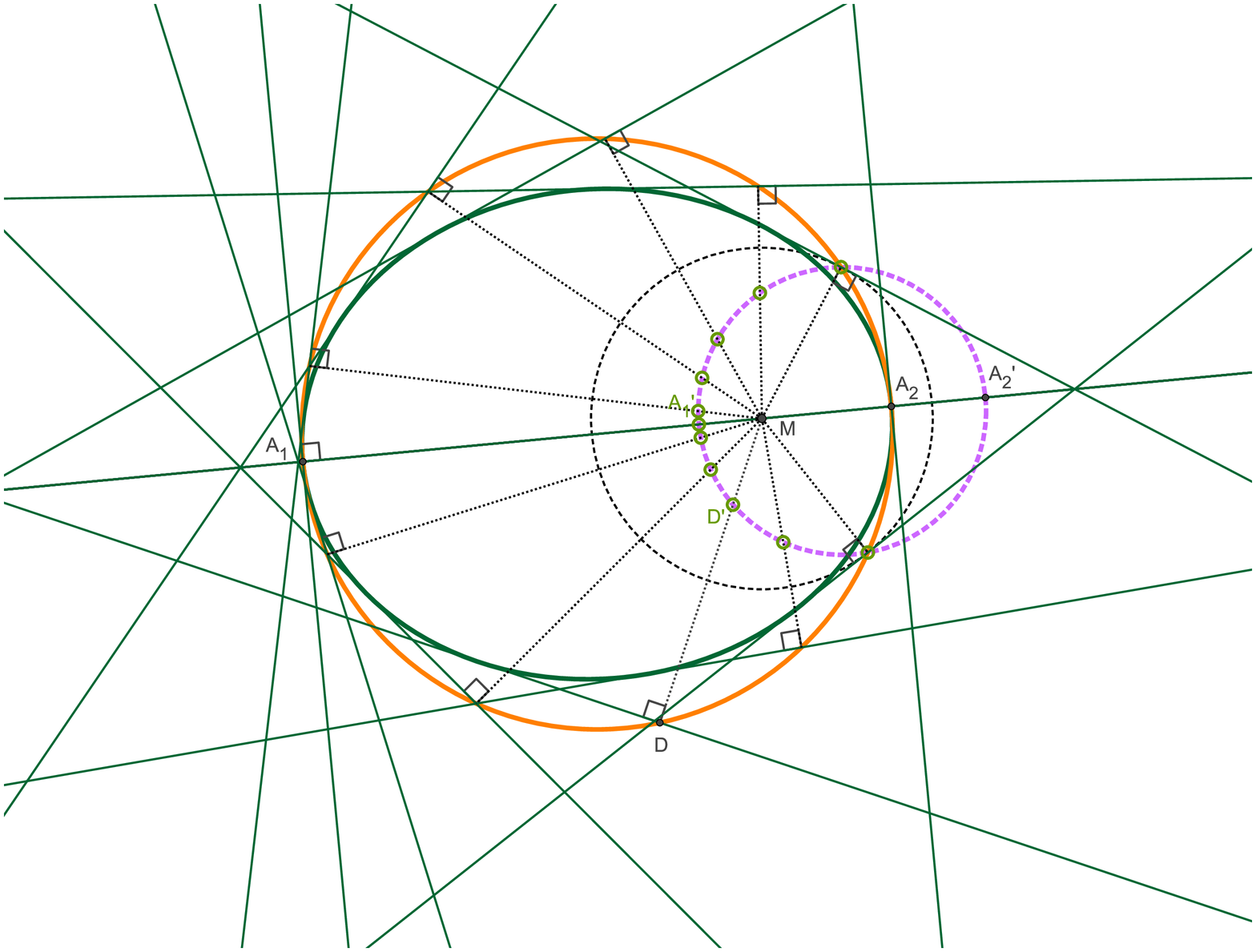}
    \caption{The negative pedal curve $\mathcal{N}(\Gamma)$ (green) of a circle $\Gamma$ (orange) w.r. to $\mathcal {I}$ (dashed black) is the envelope of lines $DE'$ (green) where $D$ is a generic point on ${\Gamma}$ and  $DE'\perp MD$.
    Since  $DE'$ is (also) the polar of $D'$ (green), the inverse of $D$, on circle $\gamma$ (dashed violet),
    then $\mathcal{N}(\Gamma)$ is  the envelope of the polars of its inverse circle $\gamma$. 
   Therefore, $\mathcal{N}(\Gamma)$  is
    the dual of its inverse circle $\gamma$,  hence  a conic with a focus at $M$.  Its vertices coincide with points $A_1,A_2$, the diameter of $\Gamma$  through $M$.
    Here, $\mathcal{N}(\Gamma)$  is an ellipse since $M$ is inside $\Gamma$.}
      \label{fig:pedal_negativa_elipse}
\end{figure}

\begin{figure}
    \centering
    \includegraphics[trim=100 90 100 10,clip,width=1.0\textwidth]
    {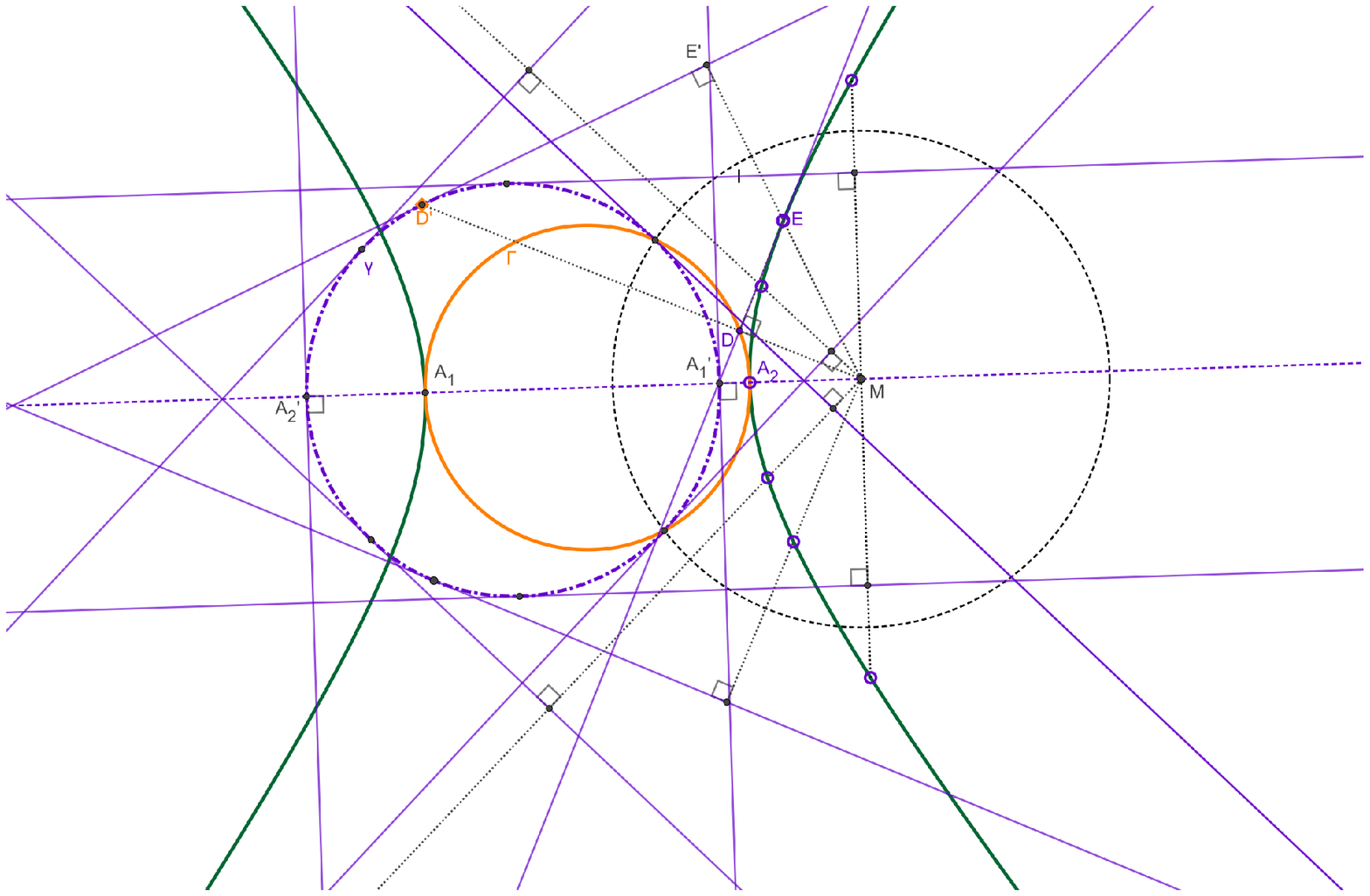}
    \caption{ $\mathcal{N}(\Gamma)$ (green hyperbola), the negative pedal curve of a circle $\Gamma$ (orange), being the dual circle $\gamma$ (violet), which is the inverse of $\Gamma$, is
    also the locus of the poles $E$ of tangents in $D'$ to $\gamma$ as $D$ sweeps $\Gamma$. The line  $DE$
is the tangent to $\mathcal{N}(\Gamma)$ passing through $D$.
   $\mathcal{N}(\Gamma)$ is an ellipse (see Fig~\ref{fig:pedal_negativa_elipse}) iff $M\in [A_1A_2]$; 
     $\mathcal{N}(\Gamma)$ is a hyperbola iff $M$ is not on segment
     $[A_1 A_2]$. The negative pedal curve of a circle is never a parabola.}
    \label{fig:pedal_negativa_hiperbole}
\end{figure}
%\begin{figure}
%    \centering
%    \includegraphics[trim=0 50 120 %0,clip,width=1.0\textwidth]{pics/0220_pedal_negativa_reta_parabola.eps}
%    \caption{The negative pedal of a line (orange) is a parabola (dark green) %whose focus is on $M,$  whose vertex is the projection of $M$ on said line and %whose 
%    directrix is  the polar of the center of  $\gamma$ (violet circle-the inverse %of the line). Also
%    illustrated is the construction of
%    the negative pedal of a line  either by definition or as  loci of poles of %the tangents to its reciprocal circle (violet)}
%    \label{fig:pedal_negativa_parabola}
%\end{figure}

These  remarkable  results are classic;
 see \cite{akopyan2007-conics} or 
 \cite[Art.309]{salmon1869}, 
for a proof and more details.

There is a natural intertwining between the negative pedal curve, inversion, and polar reciprocity.

\begin{proposition} 
The negative pedal curve of $\Gamma$ w.r. to a pedal point is the reciprocal of its inverse, $\Gamma'$ w.r. to a circle centered on that pedal point: $\mathcal{N}(\Gamma)={\Gamma'}^*$. 
Therefore (i) the negative pedal curve is the locus of the poles of the tangents 
 to its inverted curve; (ii) the polars of the points of a  negative pedal curve $\mathcal{N}(\Gamma)$ are the tangents to its inverted curve $\Gamma '$.
\label{prop:negative-pedal-point-polar}
 \end{proposition}
 
\begin{proof}
 First we prove that
 the dual of $\Gamma'$ is contained in the negative pedal of $\Gamma$. Let $S$ be a point in  $\Gamma';$ 
 then $S$ is an inverse of some point $L$ in $\Gamma$; $S=L'$; hence, the polar of $S$ is the perpendicular in point $S'$ to the line joining $M$ and $S'$; since $S'=(L')'=L,$ the polar of $S=L'$ is the perpendicular in $L$ to line $ML$. Since inversion is bijective (in fact, it is an involution), if $S$ sweeps $\Gamma'$, $L$ sweeps $\Gamma$, and  lines $ML$ are the set of all  tangent lines  to the negative pedal curve of $\Gamma$. 
 
 The reverse inclusion is similar, if we refer to negative pedal curves  as  an envelope of lines.
 \end{proof} 
  
Thus, the negative pedal curve, initially defined as an envelope of lines, can also be constructed as a ``point curve'', i.e. as the locus of the poles of the tangents to its inverse $\Gamma'$. 
 
 In order to construct the negative pedal of a circle w.r. to a pedal point that does not lie on $\Gamma$, we first draw its inverse, $\Gamma'$ then obtain the latter's dual. 
 Note that the inverse of a circle 
 can be a circle or a line. The latter case occurs if the center of inversion is on the inverted circle.
 
Below we describe the negative pedal curve of a circle. Refer to figure 
\ref{fig:pedal_negativa_hiperbole}.

\begin{proposition}
The negative pedal $\mathcal{N}(\Gamma)$ of a circle $\Gamma$,
w.r. to a pedal point $M$ not located on the boundary of $\Gamma$,
is a conic, whose (i) focus is $M$; 
\item center is the center of circle $\Gamma$; (ii) vertices are the intersection points 
of the line that joins the pedal point $M$ and the center of $\Gamma$,
 with the circle $\Gamma$; (iii) focal axis is the diameter of  $\Gamma$.

$\mathcal{N}(\Gamma)$
will be an ellipse (resp. hyperbola), if the pedal point is interior (resp. exterior) to  $\Gamma$.

The negative pedal curve of a circle,
w.r. to a point on the circumference,
reduces to a point, 
the antipode of $M$.
\label{prop:npc-locus}
\end{proposition}

\begin{proof}
First assume that $M$ is not on the circle.
Draw the negative pedal curve of the circle as follows:

\begin{enumerate}
\item
 let $\Gamma',$ be the inverse of  $\Gamma$; then 
 $\Gamma'$ is a circle whose diameter is $[A_1'A_2']$ where $A_1'$ and $A_2'$ are the inverses of vertices $A_1$ and $A_2$ of  $\Gamma$.
 \item perform the dual of $\Gamma'$ to obtain a conic whose focus is $M$
 (the inversion center) and whose vertices are  the inverses of $A_1'$ and $A_2'$, respectively, i.e., $A_1$ and $A_2$.
 \end{enumerate}
 
 Then $\mathcal{N}(\Gamma)=\big[\Gamma'\big]^*$. By the above, 
 the conic will be an ellipse (resp. hyperbola),   if $M$ is inside (resp. outside)  $\Gamma$.
 
 If the pedal point $M$  is on the circle, then the inverse is a line, whose reciprocal is a point, its pole.
 \end{proof}
 
 %\textcolor{red}{colocar as figuras de polares reciprocas}

\bibliographystyle{maa}
\bibliography{references} 

\end{document}